\documentclass[11pt]{article}

\usepackage{amsmath,amsthm,amssymb,amsfonts,mathtools,mathrsfs,enumerate}
\usepackage[usenames,dvipsnames]{color}
\usepackage{graphicx,psfrag,epsfig,ifpdf}
\usepackage{mathabx}
\usepackage{pstricks}
\usepackage{pst-bezier}
\usepackage[all]{xy}

\newcounter{parag}[subsection]

\newcounter{paraga}[subsection]
\renewcommand{\theparaga}{{\bf\arabic{paraga}.}}
\newcommand{\paraga}{\medskip \addtocounter{paraga}{1} 
\noindent{\theparaga\ } }

\newcounter{pparag}

\newtheorem{thm}{Theorem}

\newtheorem{lemma}{Lemma}[section]

\def\text#1{\,\hbox{#1}\;}

\def\al{\alpha}

\def\Ga{{\Gamma}}
\def\de{\delta}

\def\De{\Delta}
\def\eps{{\varepsilon}}
\def\ka{\kappa}
\def\la{\lambda}
\def\La{\Lambda}
\def\om{\omega}
\def\Om{\Omega}

\def\sig{{\sigma}}
\def\Sig{{\Sigma}}

\def\th{{\theta}}

\def\ph{\varphi}
\def\ze{{\zeta}}

\def\beps{{\boldsymbol \varepsilon}}

\def\bphi{{\boldsymbol \phi}}

\def\jA{{\mathscr A}}
\def\jB{{\mathscr B}}
\def\jC{{\mathscr C}}
\def\jD{{\mathscr D}}

\def\jF{{\mathscr F}}

\def\jH{{\mathscr H}}
\def\jI{{\mathscr I}}

\def\jO{{\mathscr O}}
\def\jP{{\mathscr P}}

\def\jS{{\mathscr S}}
\def\jT{{\mathscr T}}

\def\jV{{\mathscr V}}

\def\jX{{\mathscr X}}

\def\jZ{{\mathscr Z}}

\def\A{{\mathbb A}}

\def\N{{\mathbb N}}

\def\R{{\mathbb R}}
\def\T{{\mathbb T}}
\def\Z{{\mathbb Z}}

\def\Im{{\rm Im\,}}

\def\Sup{\mathop{\rm Sup\,}\limits}

\def\Max{\mathop{\rm Max\,}\limits}
\def\Min{\mathop{\rm Min\,}\limits}

\def\dist{{\rm dist\,}}

\def\setm{\setminus}

\def\ov{\overline}
\def\til{\widetilde}
\def\ha{\widehat}
\def\d{\partial}
\def\inv{^{-1}}

\def\pdemi{{\tfrac{1}{2}}}

\def\abs#1{\left\vert#1\right\vert}
\def\norm#1{\Vert#1\Vert}
\def\setm{\setminus}

\def\bS{{\mathcal S}}

\def\bS{{\mathcal S}}

\def\bH{{\mathsf H}}

\def\beq{\begin{equation}}
\def\eeq{\end{equation}}

\def\bu{{\bullet}}
\def\e{{\bf e}}

\def\bH{{\mathsf H}}

\def\Id{{\rm Id}}

\def\beq{\begin{equation}}
\def\eeq{\end{equation}}

\def\bu{{\bullet}}

\def\jC{{\mathscr C}} 

\def\H{{\bf H}} 
\def\bA{{\bf A}}

\def\bchi{{\boldsymbol \chi}}

\def\beps{{\boldsymbol\eps}}

\def\trans{\pitchfork}
 
\def\Dom{{\rm Dom\,}}
\def\codim{{\rm codim\,}}
\def\ev{{\bf ev\,}}

\def\Homt{}
\def\Hett{}
\def\Domt{}
\def\Ess{{\bf Ess\,}}

\def\cl{}
\def\bA{{\bf A}}

\def\fom{{\rm(FS)}}
\def\fomu{{\rm(FS1)}}
\def\fomd{{\rm(FS2)}}
\def\fet{{\rm(FS)}}
\def\pom{{\rm(PS)}}
\def\pomu{{\rm(PS1)}}
\def\pomd{{\rm(PS2)}}
\def\glu{{\rm(G)}}
\def\Homt{{\rm Homt}}
\def\Hett{{\rm Hett}}

\def\cH{{\mathcal H}}
\def\Ann{{\rm Ann}}

\def\dem{\De^{(-1)}}
\def\invt{^{-\tau}}
\def\ev{{\bf ev}}
\def\codim{{\rm codim\,}}

\def\cY{{\mathcal Y}}
\def\cL{{\mathcal L}}
\def\bY{{\bf Y}}
\def\Tess{{\bf Tess}}
\def\Tan{{\rm Tan}}
\def\cl{{\rm cl}}

\begin{document}

\title{Arnold diffusion for cusp-generic\\ 
nearly integrable convex systems on $\A^3$}

\author{Jean-Pierre Marco
\thanks{Universit\'e Paris 6, 
4 Place Jussieu, 75005 Paris cedex 05.
E-mail: jean-pierre.marco@imj-prg.fr
}}
\date{}

\maketitle

\begin{abstract} 
Using the results of \cite{Mar} and \cite{GM}, we prove the existence of 
``Arnold diffusion orbits'' in cusp-generic nearly integrable {\em a priori} stable systems 
on $\A^3$.

More precisely, we consider perturbed systems of the form $H(\th,r)=h(r)+f(\th,r)$, where $h$ is a 
$C^\ka$ strictly convex and superlinear  function on $\R^3$ and $f\in C^\ka(\A^3)$, $\ka\geq2$. 
We equip  $C^\ka(\A^3)$  with the uniform seminorm 
$$
\norm{f}_\ka=\sum_{k\in\N^6,\ 0\leq \abs{k}\leq \ka}\norm{\d^kf}_{C^0(\A^3)}\leq+\infty
$$
we set
$
C_b^\ka(\A^3)=\big\{f\in C^\ka(\A^3)\mid \norm{f}_\ka<+\infty\big\},
$
and we denote by $\bS^\ka$ its unit sphere. 
Given a ``threshold function'' $\beps_0:\bS^\ka\to [0,+\infty[$, we define the associated $\beps_0$-ball as
$$
\jB^\ka(\beps_0):=
\big\{\eps {\bf f} \mid {\bf f}\in\bS^\ka,\ \eps\in\,]0,\beps_0({\bf f})[\big\},
$$
so that $\jB^\ka(\beps_0)$ is open in $C_b^\ka(\A^3)$ when $\beps_0$ is lower semicontinuous.

Given $h$ as above, an energy $\e>\Min h$ and a finite family
of arbitrary open sets $O_i$ in $\R^3$ intersecting $h\inv(\e)$, we prove 
the existence of a lower semicontinuous threshold function $\beps_0$,
positive on an open dense subset of $\bS^\ka$, such that for $f$ in an open dense subset of 
the associated ball $\jB^\ka(\beps_0)$, the system $H=h+f$ admits orbits intersecting
each open set  $\T^3\times O_i\subset\A^3$.
\end{abstract}

\newpage

\section{The main results}
We denote by $\A^n=\T^n\times\R^n$ the cotangent bundle of the torus $\T^n$.
This paper is the last step of a geometrical proof of the existence of Arnold diffusion for a 
``large subset'' of perturbations of strictly convex integrable systems on $\A^3$. It relies on
the results of \cite{Mar} and \cite{GM}. Two different approaches are developped in \cite{KZ}
and \cite{C}.

\paraga Before giving a precise statement, let us quote the initial formulation of the diffusion conjecture
by V.I. Arnold, from \cite{A94}.
{\em 
``Consider a generic analytic Hamiltonian system close to an integrable one:
$$
H=H_0(p)+\eps H_1(p,q,\eps)
$$
where the perturbation $H_1$ is $2\pi$-periodic in the angle variables
$(q_1,\ldots,q_n)$ and where the nonperturbed Hamiltonian function $H_0$ depends 
on the action variables $(p_1,\ldots,p_n)$ generically. Let $n$ be greater than $2$.

\vskip2mm

{\bf Conjecture.} 
For any two points $p',p''$ on the connected level hypersurface of $H_0$ in the action space,
there exist orbits connecting an arbitrary small neighborhood of the torus $p=p'$
with an arbitrary small neighborhood of the torus $p=p''$, provided that $\eps>0$ is sufficiently 
small and that $H_1$ is generic.''
}
\vskip2mm

\paraga The formulation of the conjecture is rather imprecise and we first have to clarify our framework.
We are concerned with the case $n=3$ only and we restrict  ourselves to finitely differentiable
systems. Moreover, we adopt a setting close to that introduced by Mather in \cite{Mat04}, which we 
now describe with our usual notation. 

Let $(\th,r)$ be the angle-action coordinates on $\A^3$, and let $\la=\sum_{i=1}^3r_id\th_i$ be the
Liouville form of $\A^3$. 
The Hamiltonian systems we consider have the form
\beq\label{eq:hampert}
H(\th,r)=h(r)+f(\th,r),\qquad (\th,r)\in\A^3,
\eeq
where the unperturbed part  $h:\R^3\to\R$ is a $C^\ka$ ($\ka\geq 2$) Tonelli Hamiltonian, that is,
$h$ is strictly  convex with superlinear growth at infinity.  Here we therefore relax the genericity
condition initially imposed by Arnold.

We want to find $C^\ka$ Hamiltonians $H$  which 
admit {\em diffusion orbits} intersecting prescribed open sets in their energy level.
More precisely, we start with a Tonelli Hamiltonian~$h$ and fix an energy $\e>\Min h$,
which is therefore a regular value of $h$ whose level set $h\inv(\e)$ is diffeomorphic to
$S^2$. We then fix a finite collection of arbitrary open sets $(O_i)_{1\leq i\leq i_*}$ in $\R^3$, 
which intersect $h\inv(\e)$. Given  $H\in C^\ka(\A^3,\R)$, a diffusion orbit associated with 
these data is an orbit of the system generated by $H$ which intersects each open set 
$\T^3\times O_i\subset\A^3$.   

Our problem is to prove the existence of a ``large'' set
of perturbations $f$ for which (\ref{eq:hampert}) possesses such diffusion orbits.
We equip  $C^\ka(\A^3,\R)$  with the uniform seminorm 
\beq
\norm{f}_\ka=\Max_{0\leq \abs{k}\leq \ka}\Sup_{x\in\A^3}{\d^kf(x)}\leq+\infty,
\eeq
and we set
\beq
C_b^\ka(\A^3,\R)=\big\{f\in C^\ka(\A^3,\R)\mid \norm{f}_\ka<+\infty\big\},
\eeq
so that $\big(C_b^\ka(\A^3,\R),\norm{\ }_\infty\big)$ is a Banach  algebra. 
Let $\bS^\ka$ and $B^\ka(\rho)$ stand for the unit sphere and the ball with radius 
$\rho$ centered at~$0$ in $C_b^\ka(\A^3,\R)$. 
Given a ``threshold function'' $\beps_0:\bS^\ka\to\R$, one introduces the associated generalized ball:
\beq\label{eq:cuspball}
\jB^\ka(\beps_0):=
\big\{\eps u \mid u\in\bS^\ka,\ \eps\in\,]0,\beps_0(u)[\big\}.
\eeq
Observe that $\jB^\ka(\beps_0)$ is open in $C_b^\ka(\A^3,\R)$ when $\beps_0$ is lower-semicontinuous.

\paraga The main result of this paper is the following.

\begin{thm}\label{thm:main1}
Consider a $C^\ka$ integrable Tonelli Hamiltonian $h$ on $\A^3$.
Fix $\e>\Min h$ together with a finite family of arbitrary open sets $O_1,\ldots,O_m$
which intersect $h\inv(\e)$.  Then for $\ka\geq \ka_0$ large enough, there exists a lower-semicontinuous
function 
$$
\beps_0:\jS^\ka\to\R^+
$$ 
with positive values on an open dense subset of $\jS^\ka$ such that the subset of all $f\in\jB^\ka(\beps_0)$
for which the system
\begin{equation}
H(\th,r)=h(r)+ f(\th,r)
\end{equation}
admits an orbit which intersects each $\T^3\times O_k$ is open and dense in $\jB^\ka(\beps_0)$.
\end{thm}

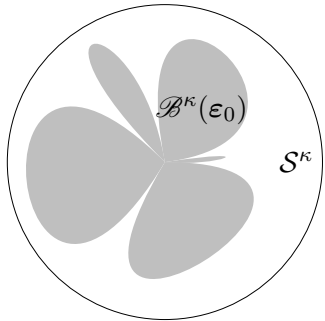
\begin{figure}[h]
\begin{center}
\begin{pspicture}(0cm,1.8cm)
\psset{xunit=.7,yunit=.7,runit=.7}
\pscircle[linewidth=0.1mm](0,0){3}
\psbezier*[linecolor=lightgray](0,0)(1,-.1)(2,.3)(0,0)
\psbezier*[linecolor=lightgray](0,0)(4,.5)(-1,5)(0,0)
\psbezier*[linecolor=lightgray](0,0)(5,0)(-3,-5)(0,0)
\psbezier*[linecolor=lightgray](0,0)(-.5,3)(-3,3)(0,0)
\psbezier*[linecolor=lightgray](0,0)(-4,4)(-3,-5)(0,0)
\rput(2.5,0){$\bS^\ka$}
\rput(.7,1){$\jB^\ka(\beps_0)$}
\end{pspicture}
\vskip20mm
\caption{A generalized ball}\label{Fig:genball}
\end{center}
\end{figure}

\paraga Our approach consists in proving first the existence of a ``geometric skeleton'' for
diffusion and then to use this skeleton to produce the diffusion orbits. Let us informally describe
our method (the precise definitions were introduced in \cite{Mar} and \cite{GM}, they are recalled 
in Section~\ref{Sec:setting}).

The main objects constituting the skeleton are $3$-dimensional invariant cylinders with boundary,
diffeomorphic to $\T^2\times[0,1]$, which are moreover normally hyperbolic and satisfy
several additional properties, and which we call {\em admissible cylinders}. 
We also have to consider admissible {\em singular} cylinders. 
These cylinders and singular cylinders contain particular $2$-dimensional invariant tori, which
we call {\em essential tori}, and admit {\em homoclinic correspondences} coming from
the homoclinic intersections of the essential tori. 


Finally, we define an  {\em admissible chain},
as a finite family $(\jC_k)_{1\leq k\leq k_*}$ of admissible
cylinders or singular cylinders, with heteroclinic connections from $\jC_k$ to $\jC_{k+1}$, 
$1\leq k\leq k_*-1$, which satisfy additional dynamical conditions.

The main result of \cite{Mar} is the following.
\vskip2mm\noindent
{\bf Theorem \cite{Mar}.} {\it 
Consider a $C^\ka$ integrable Tonelli Hamiltonian $h$ on $\A^3$.
Fix $\e>\Min h$ and a finite family of open sets $O_1,\ldots,O_m$
which intersect $h\inv(\e)$. Fix $\de>0$.  
Then for $\ka\geq \ka_0$ large enough, there exists a lower-semicontinuous
function 
$$
\beps_0:\jS^\ka\to\R^+
$$ 
with positive values on an open dense subset of $\jS^\ka$ such that for $f\in\jB^\ka(\beps_0)$
the system
\begin{equation}
H(\th,r)=h(r)+ f(\th,r)
\end{equation}
admits an admissible chain of cylinders and singular cylinders
such that each open set $\T^3\times O_k$ contains the $\de$-neighborhood
in $\A^3$ of some essential torus of the chain.}
\vskip2mm

Still, admissible chains need not admit diffusion orbits drifting along them.
In \cite{GM} is introduced the more refined notion of  {\em good
chains of cylinders}. A {\em $\de$-admissible orbit} for a good chain is an orbit which 
intersects the $\de$-neighborhood (in $\A^3$) of any essential torus  of
the chain. The main result of \cite{GM} is the following.

\vskip2mm\noindent
{\bf Theorem \cite{GM}.} {\it 
Let $H$ be a $C^2$ proper Hamiltonian on $\A^3$ and let $\e$ be a regular value of~$H$.
Then, for any good chain of cylinders contained in $H\inv(\e)$ and for any $\de>0$,
there exists a $\de$-admissible orbit for the chain.
}

\paraga Taking the previous two results for granted, Theorem~\ref{thm:main1} is an easy consequence
of the following perturbative result, whose proof constitutes the main part of the paper.

\begin{thm}\label{thm:main2}
Let $H$ be a $C^\ka$ proper Hamiltonian on $\A^3$ and let $\e$ be a regular value of~$H$.
Fix $\de>0$ and assume that $H$ admits an admissible chain $(\jC_k)_{1\leq k\leq k_*}$.
Then for any $\al>0$ there exists a Hamiltonian $\cH\in C^\ka(\A^3)$ with
\beq
\norm{H-\cH}_\ka<\al
\eeq
such that $(\jC_k)_{1\leq k\leq k_*}$ is a good chain at energy $\e$ for $\cH$,
such that each open set $\T^3\times O_k$ contains the $\de$-neighborhood
in $\A^3$ of some essential torus.
\end{thm}

\paraga We recall the necessary definitions from \cite{Mar} and \cite{GM} in Section~\ref{Sec:setting}.
We state in Section~\ref{Sec:perturbation} a perturbative result for the characteristic foliations of
the stable and unstable manifolds of a normally hyperbolic manifold, which is the main ingredient for the 
proof of Theorem~\ref{thm:main2}.
In Section~\ref{Sec:proof2} we prove Theorem~\ref{thm:main2}, from which we deduce Theorem~\ref{thm:main1}
thanks to the previous two results of \cite{Mar,GM}. We recall some necessary results on normally hyperbolic manifolds
in Appendix~\ref{app:normhyp}.


\setcounter{paraga}{0}
\section{The setting}\label{Sec:setting}
The Hamiltonian vector field associated with a $C^2$ function $H$ will be denoted by $X_H$ 
and its Hamiltonian flow, when defined, by $\Phi_H$.


\setcounter{paraga}{0}
\subsection{Normally hyperbolic annuli and cylinders}\label{sec:normhyp}

We introduce the main objects of our construction, that is, normally hyperbolic $3$-dime\-nsional cylinders 
and singular cylinders with boundary. 
We refer to \cite{C04,Berg10} for direct presentations of the normal hyperbolicity of manifolds with 
boundary.
Here we will recall the definitions of \cite{Mar} (to which we refer for more details), which take advantage of 
the existence of invariant
$4$-dimensional symplectic annuli containing the cylinders in their relative interior. These
annuli will moreover play an essential role in the definition of the intersection conditions in the next section.

\paraga A {\em $4$-annulus} will be a $C^p$ manifold $C^p$ 
diffeomorphic to $\A^2$, with $p\geq 2$.
A {\em singular annulus} will be a ($4$-dimensional) $C^1$ manifold 
$C^1$-diffeomorphic to $\T\times\,]0,1[\,\times \cY$, 
where $\cY$ is (any realization of) the sphere $S^2$ minus three points.

\paraga 
 A  {\em  $C^p$  cylinder} is a $C^p$ manifold $C^p$-diffeomorphic to $\T^2\times [0,1]$, so that a cylinder
is compact and its boundary has two components diffeomorphic to $\T^2$.
A  {\em singular cylinder} 
is a $C^1$ manifold $C^1$-diffeomorphic to $\T\times \bY$, where~$\bY$ is
(any realization of) the sphere $S^2$ minus three open discs with nonintersecting 
closures. Our singular cylinders will moreover be of class $C^p$, $p\geq 2$, in 
large neighborhoods of their boundary
(which admits three components, diffeomorphic to $\T^2$).

\paraga We endow now $\A^3$ with its standard symplectic
form $\Om$, and we assume that $X=X_H$ is the vector field generated by $H\in C^\ka(\A^3)$, $\ka\geq 2$.
Following \cite{Mar}, we say that an invariant $4$-annulus $\jA\subset A^3$ for $X$ is 
{\em normally hyperbolic} when there exist 
\vskip1.5mm
$\bu$ 
an open subset $O$ of $\A^3$ containing $\jA$, 
\vskip1.5mm
$\bu$ an embedding $\Psi:O\to \A^2\times \R^2$ whose image has compact closure, such that $\Psi_*\Om$ continues
to a symplectic form $\ov\Om$ on $\A^2\times \R^2$ which satisfies (Appendix~\ref{app:normhyp} (\ref{eq:assumpsymp})),
\vskip1.5mm
$\bu$ a vector field $\jV$ on
$\A^2\times \R^2$ satisfying the assumptions of the normally hyperbolic persistence theorem, in particular~(\ref{eq:addcond}), 
together with those of the symplectic normally hyperbolic theorem (Appendix~\ref{app:normhyp}) for the form $\ov\Om$,
such that, with the notation of this theorem:
\beq
\Psi(\jA)\subset \Ann(\jV)\quad{\rm and}\quad \Psi_*X(x)=\jV(x),\quad \forall x\in O.
\eeq
Such an annulus $\jA$ is therefore of class $C^p$ and symplectic. 
We define  normally hyperbolic singular $4$-annuli, with in this case $p=1$.
One easily checks that  annuli and singular annuli are uniformly
normally hyperbolic in the usual sense. In particular, they admit well-defined stable, 
unstable, center-stable and
center-unstable manifolds. The stable and unstable manifolds are coisotropic and their 
characteristic foliations coincide with their center-stable
and center-unstable foliations

\paraga With the same assumptions, let $\e$ be a regular value of $H$.  
Here we say for short that an invariant cylinder (with boundary) $\jC\subset H\inv(\e)$ for $X_H$  is 
{\em normally hyperbolic in $H\inv(\e)$} when there exists 
an invariant normally hyperbolic  {\em symplectic} $4$-annulus $\jA$ for $X_H$, 
such that $\jC\subset\jA\cap H\inv(\e)$.
Any such $\jA$  is said to be {\em associated with $\jC$}.

We say for short that a singular cylinder  $\jC_\bu\subset H\inv(\e)$  is invariant for $X_H$
when it is invariant together with its critical circles. We say that $\jC_\bu$
{\em normally hyperbolic in $H\inv(\e)$} when there is 
an invariant normally hyperbolic   {\em symplectic} singular annulus $\jA_\bu$ for $X_H$ such that 
$\jC_\bu\subset\jA_\bu\cap H\inv(\e)$. Any such $\jA$  is said to be {\em associated with $\jC$}.

One immediately sees that normally hyperbolic invariant cylinders or singular cylinders, contained
in $H\inv(\e)$, admit well-defined $4$-dimensional stable and unstable manifolds with boundary, 
also contained in $H\inv(\e)$, together with their center-stable and center-unstable foliations. 
The stable and unstable manifolds of the complement in a singular cylinder of its critical circles
are $C^p$.

\paraga A normally hyperbolic cylinder admits $C^1$ characteristic projections 
$\Pi^\pm:W^\pm(\jC)\to\jC$ (since the invariant manifolds $W^\pm(\jA)$ of are $C^p$ with $p\geq 2$), 
it satisfies the $\la$-lemma (as stated in \cite{GM}) and one easily proves that it admits a $\Phi_H$-invariant 
Radon measure $\mu_\jC$, positive on its open sets. It is therefore {\em tame} in the sense of \cite{GM}.
A normally hyperbolic singular cylinder admits $C^0$ characteristic projections, which are $C^1$
in a large neighborhood of its boundary, it also satisfies the $\la$-lemma 
and admits a $\Phi_H$-invariant Radon measure $\mu_\jC$, positive on its open sets.

\paraga  Let $a<b$ be fixed. Let us introduce the notation:
$$
\bA:=\bA(a,b)=\T\times[a,b],\qquad
\d_\bu\bA=\T\times\{a\},\qquad
\d^\bu\bA=\T\times\{b\}.
$$
A {\em twist section} for a cylinder $\jC\subset H\inv(\e)$
is a global $2$-dimensional transverse section $\Sig\subset \jC_\eps$,
image of an exact-symplectic embedding $j_\Sig: \bA$, such that the associated Poincar\'e return map is a twist
map in the $j_\Sig$-induced coordinates on $ \bA$. 
Denote by $\Ess(\ph)$ the set of essential invariant circles of $\ph$ (that is, whose inverse image by $j_\Sig$
is homotopic to the base). By the Birkhoff theorem, these circles are Lispchizian graphs over the base.
One requires moreover that  the boundaries of $\Sig$ are accumulation points of $\Ess(\ph)$.

Note that  $\d\jC\cap\Sig=\d_\bu\bA\cup \d^\bu\bA$.  We define $\d_\bu\jC$ and $\d^\bu\jC$ as
the components of $\d \jC$ which contain $j_\Sig(\d_\bu\bA)$ and $j_\Sig(\d^\bu\bA)$ respectively.

\paraga A generalized twist section for  a singular  cylinder is a singular $2$ annulus which admits
a continuation to a $2$-annulus, on which the Poincar\'e return map continues to a twist map 
(see \cite{Mar,GM}).


\setcounter{paraga}{0}
\subsection{\bf Intersection conditions, gluing condition, and admissible chains}
Let $H$ be a proper $C^2$ Hamiltonian function on $\A^3$ and fix a regular value~$\e$. 

\paraga{\bf Oriented cylinders.} We say that a cylinder $\jC$ is {\em oriented} when an order
is prescribed on the two components of its boundary. We denote the first one by $\d_\bu\jC$
and the second one by $\d^\bu\jC$.


\paraga {\bf The homoclinic condition \fomu.} 
A compact invariant cylinder $\jC\subset H\inv(\e)$ with twist section $\Sig$ and associated
invariant symplectic $4$-annulus $\jA$ satisfies condition \fomu\ when there exists a 
$5$-dimensional submanifold $\De\subset \A^3$, transverse to $X_H$ such that:
\begin{itemize}
\item  there exist $4$-dimensional submanifolds $\jA^\pm\subset W^\pm(\jA)\cap\De$ such that the
restrictions to $\jA^\pm$ of the characteristic projections $\Pi^\pm:W^\pm(\jA)\to\jA$ are
diffeomorphisms, whose inverses we denote by $j^\pm:\jA\to\jA^\pm$;
\item there exists a continuation $\jC_*$ of $\jC$ such that 
 $\jC^\pm_*=j^\pm(\jC_*)$ 
have a nonempty intersection, transverse in the $4$-dimensional manifold 
$
\De_\e:=\De\cap H\inv(\e)
$;
so that $\jI_*:=\jC_*^+\cap\jC_*^-$ is a $2$-dimensional submanifold of $\De_\e$;
\item the projections $\Pi^\pm(\jI_*)\subset\jC_*$ are $2$-dimensional transverse sections of the vector field $X_H$ 
restricted to $\jC_*$, and the associated Poincar\'e maps $P^\pm: \Pi^\pm(\jI_*)\to\Sig_k$ are diffeomorphisms.
\end{itemize}
We then say that 
\beq
\psi=P^+\circ\Pi^+\circ j^-\circ (P^-)\inv_{\vert\Sig}:\Sig\to \Sig_*
\eeq
(where $\Sig_*$ is a continuation of $\Sig$), is the {\em homoclinic map} attached to $\jC$. 
Note that $\psi$ is a Hamiltonian diffeomorphism on its image.


\paraga {\bf The heteroclinic condition \fomd.} 
A pair  $(\jC_0,\jC_1)$ of compact invariant oriented cylinders with twist sections 
$\Sig_0$, $\Sig_1$ and associated invariant symplectic $4$-annuli $(\jA_0,\jA_1)$
satisfies  condition \fomd\ when there exists a $5$-dimensional submanifold 
$\De\subset \A^3$, transverse to $X_H$ such that:
\begin{itemize}
\item  there exist $4$-dimensional submanifolds 
$\til \jA_0^-\subset W^-(\jA_0)\cap\De$ 
and
$\til \jA_1^+\subset W^+(\jA_1)\cap\De$ 
such that 
$\Pi_0^-{\vert \til\jA_0^-}$
and
$\Pi_1^+{\vert \til\jA_1^+}$
are diffeomorphisms on their images $\til \jA_0$, $\til \jA_1$, which we require to be neighbohoods of the 
boundaries $\d^\bu\jC_0$ and $\d_\bu\jC_1$ in $\jA_0$ and $\jA_1$ respectively,
we denote their inverses by $j_0^-$ and $j_1^+$;
\item there exist neighborhoods $\til\jC_0$ and $\til\jC_1$ of $\d^\bu\jC^0$ and $\d_\bu\jC^1$ in continuations
of the initial cylinders, such that
$\til \jC_0^-=j_0^-(\til\jC_0)$ and $\til\jC_1^+=j_1^+(\til\jC_0)$  intersect transversely 
in the $4$-dimensional manifold 
$
\De_\e:=\De\cap H\inv(\e),
$, let $\jI_*$ be this intersection;
\item the projections $\Pi_0^-(\jI_*)\subset\jC$ and $\Pi_1^+(\jI_*)\subset\jC$ are $2$-dimensional 
transverse sections of the vector field $X_H$ 
restricted to $\til \jC_0$ and $\til\jC_1$, and the Poincar\'e maps $P_0: \Pi_0^-(\jI_*)\to\Sig_0$ and $P_1: \Pi_1^-(\jI_*)\to\Sig_1$ 
are diffeomorphisms (where $\Sig_I$ stands for Poincar\'e sections in the neighborhoods $\til\jC_i$).
\end{itemize}

We then say that 
\beq
\psi=P_1\circ\Pi^+\circ j^-\circ (P_0)\inv:\Sig_0\to\Sig_1
\eeq
is the {\em heteroclinic map} attached to $\jC$ (which is not uniquely defined).


\paraga {\bf The homoclinic condition \pomu.}  
Consider an invariant cylinder $\jC\subset H\inv(\e)$ with twist section $\Sig$ and attached Poincar\'e
return map $\ph$, so that $\Sig=j_\Sig(\T\times[a,b])$, where $j_\Sig$ is exact-symplectic.
Define $\Tess(\jC)$ as the set of all invariant tori generated by the previous
circles under the action on the Hamiltonian flow (so each element of $\Tess(\jC)$ is a Lispchitzian Lagrangian
torus contained in $\jC$). The elements of $\Tess(\jC)$ are said to be {\em essential tori}.
 
\vskip2mm

We say that an invariant cylinder $\jC$ with  associated invariant symplectic $4$-annulus $\jA$ satisfies the 
{\em partial section property~\pom} when there exists a $5$-dimensional submanifold $\De\subset \A^3$, 
transverse to $X_H$ such that:
\begin{itemize}
\item  there exist $4$-dimensional submanifolds $\jA^\pm\subset W^\pm(\jA)\cap\De$ such that the
restrictions to $\jA^\pm$ of the characteristic projections $\Pi^\pm:W^\pm(\jA)\to\jA$ are
diffeomorphisms, whose inverses we denote by $j^\pm:\jA\to\jA^\pm$;
\item there exist conformal exact-symplectic diffeomorphisms 
\beq
\Psi^{\rm ann}:\jO^{\rm ann}\to \jA,\qquad \Psi^{\rm sec}:\jO^{\rm sec}\to\De_\e:=\De\cap H\inv(\e)
\eeq
where $\jO^{\rm ann}$ and $\jO^{\rm sec}$ are neighborhoods of the zero section in $T^*\T^2$ endowed with
the conformal Liouville form $a\la$ for a suitable $a>0$;
\item each  torus $\jT\in\Tess(\jC)$ is contained in some $\jC$ and the image $\Psi^{\rm ann}(\jT)$ is a Lipschitz
graph over the base $\T^2$;
\item for each such torus $\jT$, setting $\jT^\pm:=j^\pm(\jT)\subset \De_\e$, the images
$\Psi^{\rm sec}(\jT^\pm)$ are  Lipschitz graphs over the base $\T^2$.
\end{itemize}


\paraga {\bf Bifurcation condition.} See \cite{Mar} for the definitions and assumptions.
The condition we state involves the $s$-averaged system along a simple resonance circle,
it will be translated in the following as an intrinsic condition. With the notation of \cite{Mar}:
for any $r^0\in B$,  the derivative $\tfrac{d}{dr}\big(m^*(r)-m^{**}(r)\big)$ 
does not vanish. This  immediately yields transverse heteroclinic intersection properties for the intersections
of the corresponding cylinders.


\paraga {\bf The gluing condition \glu.}  A pair  $(\jC_0,\jC_1)$ of compact invariant oriented cylinders 
satisfies  condition \glu\ when they are contained in a invariant cylinder and satisfy 
\begin{itemize}
\item  $\d^\bu\jC_0=\d_\bu\jC_1$ is a dynamically minimal invariant torus that we denote by $\jT$,
\item $W^-(\jT)$ and $W^+(\jT)$ intersect transversely in $H\inv(\e)$.
\end{itemize}


\paraga {\bf \!\! Admissible chains.}\!\!
A finite family of compact invariant oriented cylinders $(\jC_k)_{1\leq k\leq k_*}$ is an {\em admissible chain} 
when each cylinder satisfies either $\fomu$ or $\pomu$ and, for  $k$ in $\{1,\ldots,k_*-1\}$,
the pair $(\jC_k,\jC_{k+1})$ satisfies either $\fomd$ or $\glu$, or corresponds to a bifurcation point.


\setcounter{paraga}{0}
\subsection{Good cylinders and good chains}\label{sec:goodcylchains}
This section is dedicated to the additional conditions introduced in \cite{GM} which produce orbits drifting along a chain.

\paraga {\bf Special twist maps.} 
We say that a twist map $\ph$ of $\bA$ is {\em special} when it does not admit any essential invariant circle with rational rotation
number and when moreover {every} element of $\Ess(\ph)\setm\d^\bu\bA$ is either the upper boundary of a
Birkhoff zone, or accumulated from below (in the Hausdorff topology) by a sequence of elements of $\Ess(\ph)$.

\paraga {\bf The homoclinic correspondence.}  We first define the {\em transverse homoclinic intersection}
of a normally hyperbolic cylinder $\jC\subset H\inv(\e)$ as the set 
\beq
\Homt(\jC)\subset W^+(\jC)\cap W^-(\jC)
\eeq
formed by the points $\xi$ such that
\beq
W^-\big(\Pi^-(\xi)\big)\trans_\xi W^+(\jC)\quad\textrm{and}\quad W^+\big(\Pi^+(\xi)\big)\trans_\xi W^-(\jC),
\eeq
where $\trans_\xi$ stands for ``intersects transversely at $\xi$ relatively to $H\inv(\e)$.''

\vskip2mm$\bu$
Assume that  $\jC$ admits a twist section $\Sig=j_\Sig(\bA)$ and identify
$\bA$ with $\Sig$.
A {\em homoclinic correspondence} associated with these data is a family
of $C^1$ local diffeomorphisms of $\Sig$:
\beq
\psi=(\psi_i)_{i\in I},  \qquad \psi_i:\Dom \psi_i\subset\Sig\to \Im\psi_i \subset\Sig,
\eeq
where $\Dom \psi_i$ and $\Im\psi_i$ are open subsets of $\Sig$, for which there exist
is a family of $C^1$ local diffeomorphisms of $\jC$:
\beq
S=(S_i)_{i\in I},  \qquad S_i:\Dom S_i\subset\jC\to \Im S_i\subset\jC
\eeq
where $\Dom S_i$ and $\Im S_i$ are open subsets of $\jC$, such that for all $i\in I$:
\begin{itemize}
\item  there exists a $C^1$ function $\tau_i:\Dom\psi_i\to \R$ such that
\beq
\forall x\in\Dom\psi_i,\qquad \Phi_H^{\tau_i(x)}(x)\in \Dom S_i \quad \textit{and}\quad
\psi_i(x)=S_i\Big(\Phi_H^{\tau_i(x)}(x)\Big);
\eeq
\item there is an open subset $\Domt S_i\subset \Dom S_i$, with full measure in $\Dom S_i$,
such that 
\beq\label{eq:redhom3}
\forall y\in \Domt S_i,
\quad
W^-(y)\cap 
W^+\big(S_i(y)\big)\cap 
\Homt(\jC)\neq\emptyset.
\eeq
\end{itemize}
We say that a family $S$ satisfying the previous properties is {\em associated} with $\psi$.

\vskip2mm

Homoclinic correspondences are not uniquely defined and the domains $\Dom \psi_i$ (resp. $\Dom S_i$)
are not necessarily pairwise disjoint. The index set $I$ is non countable in general. In the following we 
indifferently consider our homoclinic correspondences as defined on $\Sig$ or on $\bA$.
The following additional definition is necessary to produce $\de$-admissible orbits.

\vskip2mm$\bu$
Let $\jC$ be a good cylinder with twist section $\Sig=j_\Sig(\bA)$
and a homoclinic correspondence $\psi:\bA\righttoleftarrow$. Fix $\de>0$.
We say that $\psi$ is {\em $\de$-bounded} when for each essential circle $\Ga\in\Ess(\ph)$
{
\beq
\dist\Big(\cl(\Ga),\cl\big(\Ga^-\cap\psi\inv(\Ga^+)\big)\Big)<\de
\eeq
where $\dist$ stands for the Hausdorff distance.}

\vskip2mm 
$\bu$ The previous definitions, in their full generality, will apply to cylinders which satisfy Condition \pomu,
in the sense that one immediately shows that any such cylinder admits a homoclinic correspondence.
When a cylinder $\jC$ satisfies \fomu, it turns out that the situation is much simpler: there exists a single 
$C^1$ diffeomorphism $\psi:\bA\righttoleftarrow$ which satisfy the previous compatibility condition with a
single diffeomorphism $S:\Dom S\to\Im S$. In the case of a singular cylinder $\jC_\bu$ with generalized section 
$\Sig_\bu\sim\bA_\bu$, there also exist
a single  $C^1$ diffeomorphism $\psi_\bu:\bA_\bu\righttoleftarrow$ which satisfy the previous compatibility condition with a
single diffeomorphism $S:\Dom S\to\Im S$. The diffeomorphism $\psi_\bu$ is continuable to a diffeomorphism of $\psi$
the continuation $\bA$ of the section.

\paraga {\bf Splitting arcs.} The notions we introduce now are useful only in the case of cylinders
satisfying \pomu.
We consider a such a normally hyperbolic cylinder $\jC$ equipped with a twist section $\Sig=j_\Sig(\bA)$
and a homoclinic correspondence  $\psi=(\psi_i)_{i\in I}$ on $\bA$. An {\em arc} of $\bA$ is a continuous map
$\ze=[0,1]\to\bA$. We write $\ha \ze=\ze([0,1])\subset\bA$ for the image of the arc.
Given two distinct points $\th,\th'$ of $\T$, we write $[\th,\th']$ for the (unique) segment bounded
by $\th$ and $\th'$ according to the natural orientation of $\T$.
When two points $\al=(\th,r),\al'=(\th',r')$ belong to a circle $\Ga$ which is a graph over $\T$,
we write  $[\al,\al']_\Ga$ for the oriented segment of $\Ga$ located over $[\th,\th']$, equipped with
the natural orientation of $\Ga$. We write $-[\al,\al']_\Ga$ for the segment equipped with
the opposite orientation.

\vskip2mm\noindent
Consider $\Ga\in\Ess(\ph)$ and let $\al\in\Ga$.
\begin{itemize}
\item A {\em splitting arc} based at $\al$ for the pair $(\ph,\psi)$ is an arc $\ze$ of $\bA$ for which
$$
\ze(0)=\al,\quad \ze(]0,1])\subset \Ga^-;\quad \exists i\in I,\ \ze(]0,1])\subset \Dom\psi_i,\quad \psi_i(\ze(]0,1]))\subset \Ga.
$$
\item A {\em splitting domain} based at $\al$ for the pair $(\ph,\psi)$ is a the interior of a $2$-dimensional submanifold
with boundary of $\bA$ which is contained in $\Ga^-$ and whose boundary contains a splitting arc based at $\al$;
\item A {\em simple splitting arc} based at $\al=(\th,r)$ for the pair $(\ph,\psi)$
is a splitting arc $\ze$ based at $\al$ such that
$\ha\ze$ projects over an interval $[\th,\th+\sig]$ or an interval $[\th-\sig,\th]$, with $0<\sig<\pdemi$.
\end{itemize}

\paraga {\bf Good cylinders.} We will distinguish between three cases.
\vskip1mm 1) We say that a normally hyperbolic cylinder $\jC$ which satisfies \fomu\ 
is a good cylinder when it admits a twist section $\Sig=j_\Sig(\bA)$ with return map $\ph:\bA\righttoleftarrow$
and a homoclinic map $\psi:\bA\righttoleftarrow$ such that
no element of $\Ess(\ph)$ is invariant under $\psi$.

\vskip1mm 2) We say that a normally hyperbolic singular cylinder $\jC_\bu$ which satisfies \fomu\ 
is a good cylinder when it admits a generalized twist section $\Sig_\bu=j_\Sig(\bA\bu)$ with return map 
$\ph:\bA\bu\righttoleftarrow$ (continuable as a twist map of an annulus)
and a homoclinic map $\psi:\bA\righttoleftarrow$ such that
no element of $\Ess(\ph)$ is invariant under $\psi$.

\vskip1mm 3)
We say that a normally hyperbolic cylinder $\jC$ which satisfies \pomu\ 
is a good cylinder when it admits a twist section $\Sig=j_\Sig(\bA)$ with return map $\ph:\bA\righttoleftarrow$
and a homoclinic correspondence $\psi:\bA\righttoleftarrow$ such that
\begin{itemize}
\item for any element $\Ga$ of $\Ess(\ph)$ {which is not the upper boundary of a Birkhoff zone},
 there exists a splitting domain based on $\Ga$;
\item if $\Ga\in\Ess(\ph)$ is the upper boundary of a Birkhoff zone, then there exists  a simple splitting arc
based on $\Ga$.
\end{itemize}

\paraga {\bf Heteroclinic maps.}  
Let $\jC_1$ and $\jC_2$ be disjoint good cylinders at energy $\e$ for $H$, with characteristic projections
$\Pi_i^\pm:W^\pm(\jC_i)\to\jC_i$ and twist sections $\Sig_i=j_{\Sig_i}(\bA_i)$, with $\bA_i=\T\times[a_i,b_i]$.
We define the {\em transverse heteroclinic intersection of $\jC_1$ and $\jC_2$} as the set
\beq
\Hett(\jC_1,\jC_2)\subset W^-(\jC_1)\cap W^+(\jC_2)
\eeq
formed by the points $\xi$ such that
\beq
W^-\big(\Pi_1^-(\xi)\big)\trans_\xi W^+(\jC_2)\quad\textrm{and}\quad W^+\big(\Pi_2^+(\xi)\big)\trans_\xi W^-(\jC_1).
\eeq
A {\em heteroclinic map} from $\jC_1$ to $\jC_2$ is a $C^1$ diffeomorphism 
\beq
\psi_1^2:\Dom \psi_1^2\subset \Sig_1 \to \Im\psi_1^2\subset \Sig_2
\eeq
where $\Dom \psi_1^2$ is
an open neighborhood  of $\d^\bu\Sig_1$ in $\Sig_1$ and $\Im\psi_1^2$ is an open neighborhood of $\d_\bu \Sig_2$
in $\bA_2$, for which there exists a $C^1$ diffeomorphism 
\beq
S_1^2:\Dom S_1^2\subset \jC_1 \to \Im S_1^2\subset \jC_2
\eeq
where $\Dom S_1^2$ and $\Im S_1^2$ are open subsets, which satisfies the following conditions:
\begin{itemize}
\item  there exists a $C^1$ function $\tau:\Dom\psi_1^2\to \R$ such that
\beq
\forall x\in\Dom\psi_1^2,\qquad \Phi_H^{\tau(x)}(x)\in \Dom S_1^2 \quad \textit{and}\quad
\psi_1^2(x)=S_1^2\Big(\Phi_H^{\tau(x)}(x)\Big);
\eeq
\item there is an open subset $\Domt S_1^2\subset \Dom S_1^2$, with full measure in $\Dom S_1^2$,
such that 
\beq\label{eq:compahetero}
\forall y\in \Domt S_1^2,\qquad  W^-(y)\cap W^+\big(S_1^2(y)\big)\cap\Hett(\jC_1,\jC_2)\neq\emptyset.
\eeq
\end{itemize}

\paraga {\bf Bifurcation maps.}  
Let $\jC_1$ and $\jC_2$ be disjoint good cylinders at energy $\e$ for $H$, with characteristic projections
$\Pi_i^\pm:W^\pm(\jC_i)\to\jC_i$ and twist sections $\Sig_i=j_{\Sig_i}(\bA_i)$, with $\bA_i=\T\times[a_i,b_i]$.
A bifurcation map from $\jC_1$ to $\jC_2$ is a $C^1$ diffeomorphism 
\beq
\psi_1^2:\Ga_1\subset \Sig_1 \to \Ga_2\subset \Sig_2
\eeq
where $\Ga_i$ are dynamically minimal circles for the return maps $\ph_i$, such that, denoting
by $\jT_i$ the essential tori they generate, there exists a $C^1$ diffeomorphism 
\beq
S_1^2:\jT_1 \to \jT_2
\eeq
which satisfies the following conditions:
 there exists a $C^1$ function $\tau:\jT_1\to \R$ such that
\beq
\forall x\in\jT_1,\qquad \Phi_H^{\tau(x)}(x)\in \jT_1 \quad \textit{and}\quad
\psi_1^2(x)=S_1^2\Big(\Phi_H^{\tau(x)}(x)\Big);
\eeq
and
\beq
\forall y\in \jT_1,\qquad  W^-(y)\cap W^+\big(S_1^2(y)\big)\cap\Hett(\jC_1,\jC_2)\neq\emptyset.
\eeq

\paraga {\bf Good chains.} 
 A {\em good chain of cylinders} at energy $\e$ is an admissible chain
$(\jC_k)_{1\leq k\leq k_*}$ of {\em good} cylinders or singular cylinders at energy $\e$, with twist sections $\Sig_k$, 
such that for $1\leq k\leq k_*-1$:
\begin{itemize}
\item either $\jC_k$ and $\jC_{k+1}$ are consecutive cylinders contained in the same cylinder, that is $\d_\bu\jC_k=\d^\bu\jC_{k+1}$,
which satisfy the gluing condition \glu;
\item or there exists  a bifurcation map  $\psi_k^{k+1}$ from $\jC_k$ to $\jC_{k+1}$;
\item or a heteroclinic map $\psi_k^{k+1}$ from $\jC_k$ to $\jC_{k+1}$
and a circle $\Ga_k\in\Ess (\ph_k)$ contained in $\Dom \psi_k^{k+1}$ whose image $\psi_k^{k+1}(\Ga_k)$
is a dynamically minimal essential invariant circle for $\ph_{k+1}$.
\end{itemize}


\section{Perturbation of characteristic foliations}\label{Sec:perturbation}
We now introduce the main ingredient of our perturbative construction.


\subsection{A perturbative lemma for Poincar\'e maps}
We refer to \cite{MS} for the necessary definitions and results in symplectic geometry.
We begin with a global form of the Hamiltonian flow-box theorem. 

\begin{lemma}\label{lem:hamflowbox}
Let $(M^{2m},\Om)$ be a symplectic manifold with Poisson bracket $\{\, ,\, \}$, and fix a Hamiltonian 
$H\in C^\infty(M)$, with complete
vector field. Let $\La$ be a codimension 1 submanifold of $M$, transverse to $X_H$, such that 
there exists un open interval $I\subset\R$ containing $0$ for which the restriction of $\Phi_H$ to $I\times \La$ is 
an embedding. Set 
\beq
\jD:=\Phi_H(I\times\La)
\eeq
 and let $F:\jD\to\R$ be the $C^\ka$ (transition time) function defined by
\beq
\Phi_H\big(-F(x),x\big)\in\La,\qquad \forall x\in \jD.
\eeq
Then $\{H,F\}=1$ and $\La=F\inv(0)$, so $X_F$ is tangent to $\La$. 
Assume moreover that there exist an open interval $J$ and $\ov\e\in J$ such that, setting
$\La_{\ov\e}=H\inv(\ov \e)\cap \La$,
the flow of $X_F$ is defined on $J\times \La_{\ov\e}$ and satisfies
\beq
\La=\Phi_F\big(J\times \La_{\ov\e}\big).
\eeq
Then the form $\Om_{\ov \e}$ induced by $\Om$ on $\La_{\ov\e}$ is symplectic, and the map
\beq
\begin{array}{lccll}
\bchi:  &(I\times J)\times \La_{\ov\e} &\longrightarrow& \jD&\\[5pt]
           &\big((t,\e), x\big)&\longmapsto &\Phi_H\big(t,\Phi_F(\e,x)\big)&
\end{array}           
\eeq
is a $C^\infty$ symplectic diffeomorphism on its image, where $(I\times J)\times \La_{\ov\e}$ is equipped
with the form
\beq
d\e\wedge dt\,\oplus\, \Om_{\ov \e}.
\eeq
Moreover
\beq
H\circ\bchi\big((t,\e), x\big)=\e,\qquad 
\bchi^*(X_H)=\frac{\d}{\d t}.
\eeq
\end{lemma}

In the following we say that a submanifold $\La$ satisfying the assumptions of the previous lemma 
is a {\em box-section} for $X_H$, with associated data $(I,J,\ov \e)$. Given {\em any} transverse section
$\ha\La$ of $X_H$ and a compact subset $K$ of $\ha\La\cap H\inv(\ov\e)$, one easily proves that
there exists a box-section $\La\subset \ha\La$ which is a neihborhood of $K$ in $\ha\La$.

\begin{lemma}\label{lem:perturblemma}
Let $M^{2m}$ be a symplectic manifold and fix a Hamiltonian $H\in C^\infty(M)$ with complete 
vector field $X_H$ and flow $\Phi_H$. 
Assume that $\La$ 
is a {\em box-section} for $X_H$, with associated data $(I,J,\ov \e)$ such that $[-1,0]\subset I$, set
\beq
\dem=\Phi_H\inv(\La),\qquad  \La_{\ov\e}=\La\cap H\inv(\ov\e),\qquad  \dem_{\ov\e}=\dem\cap H\inv(\ov\e),
\eeq
and let $P_H$ be the $\Phi_H$-induced Poincar\'e map between $\dem_{\ov\e}$ and $\La_{\ov\e}$.

Let $K$ be a compact of $\La_{\ov \e}$ contained in the relative interior of $\La_{\ov \e}$.
Then for any $C^\infty$ Hamiltonian diffeomorphism $\phi:\La_{\ov \e}\righttoleftarrow$ with support  in $K$,
there exists a Hamiltonian $\cH\in C^\infty(M)$ such that:
\vskip1mm
$\bu$ $\La$ is a box-section of $X_{\cH}$ with associated data $(I,J,\ov \e)$, and $\Phi_\cH\inv(\La)=\dem$,
\vskip1mm
$\bu$ $\cH$  coincides with $H$ outside $\Phi_H(]-1,0[\times \La)$,
\vskip1mm
$\bu$  the $\Phi_\cH$-induced Poincar\'e map between $\dem_{\ov\e}$ and $\La_{\ov\e}$ satisfies
\beq
P_\cH=\phi\circ P_H.
\eeq
\vskip0mm
$\bu$ $\cH$ tends to $H$ in the $C^\infty$ topology when  
$\phi\to \Id$ in the $C^\infty$ topology.
\end{lemma}

\begin{proof} 
The map
\beq
\bchi:I\times J\times \La_\e\longrightarrow \Phi_H(I\times \La)
\eeq
of  Lemma~\ref{lem:hamflowbox} is a symplectic diffeomorphism such that $H\circ\bchi(t,\e,x)=\e$.
We first work in the coordinates $(t,\e,x)$ to construct our new Hamiltonian.
Let $\ell: I\times \La_\e$ be a $C^\infty$ nonautonomous Hamiltonian on $\La_\e$ with support in
$[-2/3,-1/3]\times K$, whose associated transition map between the times $\{-2/3\}$
and $\{-1/3\}$ coincides with $\phi$.
We set
\beq
\bH(t,\e,x)=\e+\eta(\e)\ell(t,x)
\eeq
where $\eta:J\to \R$ is a $C^\infty$ function with support in $J$, equal to $1$ in an open neighborhood
$J_*$ of $\e$, so that the function $\eta\,\ell$ has compact support contained in 
$[-2/3,-1/3]\times I\times K$. The associated vector field reads, for $(t,\e,x)\in I\times J_*\times \La_\e$:
\beq
\dot t=1,\quad \dot\e =\d_t\ell(t,x),\quad \dot x=X_\ell(t,x),
\eeq
where $X_\ell$ stands for the vector field generated by $\ell$ relatively to the induced form $\Om_{\ov \e}$
on~$\La_\e$. As a consequence, the Poincar\'e map of $\bH$ between the sections $\{-1\}\times(J\times \La_\e)$
and $\{0\}\times(J\times \La_\e)$ reads:
\beq
P_\bH(-1,\e,x)=\big(0,\e,\phi(x)\big).
\eeq
The function $\bH\circ\bchi\inv$ coincides with $H$ in the open set
\beq
\bchi(I\times J\times \La_\e)\ \setm\ \bchi([-2/3,-1/3]\times {\rm Supp}\,\eta \times K)
\eeq
and so continues as a $C^\infty$ function $\cH$ on $M$, which coincides with $H$ outside 
the latter factor.
Since $\bchi$ is symplectic, $\La$ is a box-section of $X_{\cH}$ with associated data $(I,J,\ov \e)$, 
and $\Phi_\cH\inv(\La)=\dem$.
The Poincar\'e maps $P_\bH$ and $P_\cH$ satisfy 
\beq
P_\cH=\bchi\circ P_\bH\circ\bchi\inv_{\vert \Phi_H\inv(\La_{\ov\e})},
\eeq
and
\beq
\bchi(0,\ov\e,x)=x,\qquad \bchi(-1,\ov\e,x)=\Phi_H\inv(x),
\eeq
hence, setting $z=\Phi_H\inv(x)$ for $x\in \La_{\ov\e}$:
\beq
P_\cH(z)=\bchi\circ P_\bH\circ\bchi\inv(z)=\bchi\circ P_\bH(-1,\ov\e,x)=\bchi\big(0,\e,\phi(x)\big)=\phi(x)
\eeq
so that
\beq
P_\cH=\phi\circ P_H.
\eeq
Finally one can choose $\ell$ so that $\ell\to 0$ in the $C^\infty$ topology when $\phi$ tends to $\Id$ in
the $C^\infty$ topology, from which our last assertion easily follows.
\end{proof}


\subsection{Perturbations of homoclinic maps in the \fom\ case}
In this section we consider a Hamiltonian $H\in C^\ka(\A^3)$ and fix a regular value $\e$ of $H$.

\begin{lemma}\label{lem:pertfomhom}
Assume that $\jC\subset H\inv(\e)$ with twist section $\Sig$ satisfies condition \fom, and let $\psi_H$ be
its homoclinic map.
Then for any $C^{\ka-2}$ Hamiltonian diffeomorphism $\sig:\Sig\righttoleftarrow$, there exists a $C^\ka$ Hamiltonian 
$\cH$ which coincides with $H$ in the neighborhood of $\jC$, such that $\jC$ still satisfies condition \fom\
for $\cH$ and that the associated homoclinic map $\psi_\cH$ satisfies
\beq
\psi_\cH=\sig\circ\psi_H;
\eeq
moreover $\norm{\cH-H}_\ka\to 0$ when $d_{\ka-2}(\sig,\Id)\to0$.
\end{lemma}

\begin{proof}   
By the $C^\infty$-smoothing technique of \cite{Ze76}, there exists a Hamiltonian $H_*$, arbitrarily
close to $H$ in the $C^\ka$ topology and coincides with $H$ in the neighborhood of $\jC$, such 
that $\jC$ still satisfies condition \fom\ with respect to $\De$ and such that $H_*$ is $C^\infty$
in the neighborhood of $\De$. Let $\psi_{\cH_*}$ be the associated homoclinic map and set
\beq
\sig_*=\sig\circ\psi_H\circ\psi_{H_*}\inv,
\eeq
so that $\sig_*$ is a Hamiltonian diffeomorphism of $\Sig$, arbitrarily close to $\sig$ in the
$C^{\ka-2}$ topology.
It is therefore enough to prove the result for $H_*$ and $\sig_*$ instead of $H$ and~$\sig$. So 
one can assume without loss of generality that $H$ is $C^\infty$ in the neighborhood of $\De$.

By compactness of $\jC^-$,  one can find a neighborhood $\La$ of $\jC^-$ in $\De$ which is a 
box-section for $X_H$, with data $(I,J,\e)$, where $I$ contains some interval $[-\tau,0]$. One can
moreover assume that $H$ is $C^\infty$ in the neighborhood of $\Phi_H(\ov I\times \ov\La)$.

Set
\beq
\phi=(P_H^+\circ\Pi_H^+)\inv\circ\sig\circ(P_H^+\circ\Pi_H^+)_{\vert \jI}
\eeq
where $\jI=\jC^+\cap\jC^-$, so that $\phi$ is a Hamiltonian diffeomorphism of $\jI$. We proved
in \cite{Mar16} the existence of a Hamiltonian diffeomorphism $\bphi$ of $\De_\e$ which continues
$\phi$. 

By Lemma~\ref{lem:perturblemma}, there exists a Hamiltonian $\cH\in C^\ka(M)$ such that:
\vskip1mm
$\bu$ $\La$ is a box-section of $X_{\cH}$ with associated data $(I,J,\ov \e)$, and $\Phi_\cH\inv(\La)=\dem$,
\vskip1mm
$\bu$ $\cH$  coincides with $H$ outside $\Phi_H(]-1,0[\times \La)$,
\vskip1mm
$\bu$  the $\Phi_\cH$-induced Poincar\'e map between $\La^{-}_{\ov\e}$ and $\La_{\ov\e}$ satisfies
\beq
P_\cH=\phi\circ P_H.
\eeq
\vskip0mm
$\bu$ $\cH$ tends to $H$ in the $C^\ka$ topology when  
$\phi\to \Id$ in the $C^{\ka-2}$ topology.
\vskip1mm
\noindent
Fix $\xi\in\jA^-$ and set $\eta=\Phi_H\invt(\xi)$. 
Since $H$ and $\cH$ coincide outside the perturbation box
\beq
\Pi_H^-(\eta)=\Pi_\cH^-(\eta).
\eeq
By equivariance of the unstable foliation
of $W^-(\jA)$:
\beq
\Pi_\cH^-\big(\Phi_\cH(\eta)\big)=\Pi_H^-\big(\Phi_H(\eta)\big).
\eeq
Moreover, if $\xi\in\jC^-$,
\beq
\Phi_\cH^\tau(\eta)=\bphi\circ\Phi_H^\tau(\eta)
\eeq
which proves that
\beq
\Pi_\cH^-\circ\bphi(\xi)=\Pi_H^-(\xi),\qquad \forall \xi\in\jC^-.
\eeq
In particular, since $\bphi$ leaves $\jI$ invariant
\beq
\Pi_\cH^-(\jI)=\Pi_H^-(\jI).
\eeq
and the Poincar\'e maps $P_\cH^-$ and $P_H^-$ coincide. Moreover, since the perturbation
does not affect the stable manifold $W^+(\jA)$:
\beq
\psi_\cH(x)=P_H^+\circ\Pi_H^+\circ j_{\cH}^-\circ (P_H^-)\inv(x),\qquad \forall x\in \Sig.
\eeq
Hence
\beq
\psi_\cH(x)=P_H^+\circ\Pi_H^+\circ\phi\circ j_H^-\circ(P_H^-)\inv(x)=\sig\circ\psi_H(x),\qquad \forall x\in \Sig,
\eeq
which proves our claim. Finally, since $d_{\ka-2}(\bphi,\Id)\to 0$ when $d_{\ka-2}(\sig,\Id)\to 0$ (see \cite{Mar16}), 
the last statement comes from Lemma~\ref{lem:perturblemma}.
\end{proof}

The following lemma for the heteroclinic condition \fom\ is proved exactly in the same way as 
the previous one.

\begin{lemma}
Assume that the pair $(\jC_0,\jC_1)$ of compact invariant cylinders with twist sections 
$\Sig_i$, contained in $H\inv(\e)$, satisfies  condition \fet, and let $\psi_H$ be
its heteroclinic map.
Then for any $C^{\ka-2}$ Hamiltonian diffeomorphism $\sig:\Sig_2\righttoleftarrow$, there exists a 
$C^\ka$ Hamiltonian  $\cH$ which coincides with $H$ in the neighborhood of $\jC_0$ and $\jC_1$,
such that $(\jC_0,\jC_1)$  still satisfies condition \fet\ for $\cH$ and that the associated heteroclinic 
map $\psi_\cH$ satisfies
\beq
\psi_\cH=\sig\circ\psi_H;
\eeq
moreover $\norm{\cH-H}_\ka\to 0$ when $d_{\ka-2}(\sig,\Id)\to0$.
\end{lemma}


\subsection{Perturbations of characteristic projections in the \pom\ case}
We consider a Hamiltonian $H\in C^\ka(\A^3)$ and fix a regular value $\e$ of $H$.

\begin{lemma}\label{lem:perturbpom}
Assume that $\jC\subset H\inv(\e)$ with twist section $\Sig$ satisfies condition \pom. Then there
exists a compact neighborhood $K$ of $\jC^-$ in $\De_\e$ such that for any $C^{\ka-2}$ Hamiltonian 
diffeomorphism $\phi$ of $\De_\e$ with support in $K$, there exists a  $C^\ka$ Hamiltonian 
$\cH$ which coincides with $H$ in the neighborhood of $\jC$, such that $\jC$
still satisfies condition \pom\ for $\cH$ and that the associated characteristic projection satisfies
\beq
(\Pi^-_\cH)_{\vert \jC^-}=(\Pi^-_\cH)_{\vert \jC^-}\circ\phi;
\eeq
moreover $\norm{\cH-H}_\ka\to 0$ when $d_{\ka-2}(\phi,\Id)\to0$.
\end{lemma}

\begin{proof}
The proof is follows the same lines as that of Lemma~\ref{lem:perturblemma}. One can assume that $H$ 
is $C^\infty$ in  the neighborhood of $\De$.  One first constructs a box-section $\La\subset\De_\e$ with
data $(I,J,\e)$,  where $I$ contains some interval $[-\tau,0]$, such that moreover
$H$ is $C^\infty$ in the neighborhood of $\Phi_H(\ov I\times \ov\La)$.
Then, by Lemma~\ref{lem:perturblemma}, there exists a Hamiltonian $\cH\in C^\ka(M)$ such that:
\vskip1mm
$\bu$ $\La$ is a box-section of $X_{\cH}$ with associated data $(I,J,\ov \e)$, and $\Phi_\cH\inv(\La)=\dem$,
\vskip1mm
$\bu$ $\cH$  coincides with $H$ outside $\Phi_H(]-1,0[\times \La)$,
\vskip1mm
$\bu$  the $\Phi_\cH$-induced Poincar\'e map between $\La^{-}_{\ov\e}$ and $\La_{\ov\e}$ satisfies
\beq
P_\cH=\phi\circ P_H.
\eeq
\vskip0mm
$\bu$ $\cH$ tends to $H$ in the $C^\ka$ topology when  
$\phi\to \Id$ in the $C^{\ka-2}$ topology.
\vskip1mm
\noindent One then deduces exactly as in Lemma~\ref{lem:perturblemma} that
\beq
(\Pi^-_\cH)_{\vert \jC^-}=(\Pi^-_\cH)_{\vert \jC^-}\circ\phi;
\eeq
and the last statement is immediate.
\end{proof}


\section{Proofs of Theorem~\ref{thm:main2} and Theorem~\ref{thm:main1}}\label{Sec:proof2}

We now use the results of the previous section and prove that admissible chains can be made
good chains by arbitrarily small perturbations of the Hamiltonian, which is the content of 
Theorem~\ref{thm:main2}. We then prove Theorem~\ref{thm:main1}, which relies on the
results of \cite{Mar,GM}.


\subsection{Cylinders with condition \fom}
We consider a proper Hamiltonian $H\in C^\ka(\A^3)$ and fix a regular value $\e$ of $H$.

\begin{lemma}\label{lem:goodcylfom}
Assume that $\jC\subset H\inv(\e)$ satisfies condition~\fomu.
Then for any $\al>0$, there exists a $C^\ka$ Hamiltonian  $\cH$
such that $\jC$ is a good cylinder for $\cH$, and which satisfies
\beq
\norm{\cH-H}_\ka<\al.
\eeq
\end{lemma}

\begin{proof} By \cite{R} applied to the symplectic manifold $\jA$, there exists an arbitrarily 
small perturbation $\til H$ of $H$ in the $C^\ka$ topology which admits an invariant annulus $\til \jA$
and an invariant cylinder $\til\jC$ which are arbitrarily $C^\ka$-close to the initial ones,
such that  any periodic orbit of $\til H$ is hyperbolic in $\jC$ or elliptic 
with nondegenerate torsion and KAM nonresonance conditions. 

As a consequence, one can choose $\til H$ such that $\norm{H-\til H}_{C^\ka}<\al/2$, 
$\til\jC$ still satisfies 
condition \fomu\ and admits a section $\Sig$  for which the Poincar\'e
map $\til\ph$ is a special twist map. We can therefore assume that these properties
are satisfied by the initial Hamiltonian $H$ and we get rid of the $\til{\phantom{u}}$.

Let $\psi$ be the homoclinic map of $\jC$ for $H$.
By \cite{M02}, there exists a $C^\ka$ Hamiltonian diffeomorphism $\sig$ of $\Sig$,
arbitrarily close to $\Id$, such that the return map  $\ph:\Sig\righttoleftarrow$ and the map 
$\sig\inv\circ\psi\circ\sig:\Sig\righttoleftarrow$
have no common essential invariant circle. Observe that $[\sig\inv\circ\psi\circ\sig\circ\psi\inv]$ 
is a Hamiltonian diffeomorphism of $\Sig$, which tends to $I$ in the $C^\ka$ topology when 
$\sig$ tends to $\Id$ in the $C^\ka$ topology. So,
by Lemma~\ref{lem:pertfomhom}, if $\sig$ is
close enough to the indentity, there
exists a Hamiltonian $\cH$ which coincides with $H$ in the neighborhood of $\jC$,
such that $\jC$ still satisfies condition~\fomu, with new homoclinic map 
\beq
\psi_{\cH}=[\sig\inv\circ\psi\circ\sig\circ\psi\inv]\circ\psi=\sig\inv\circ\psi\circ\sig.
\eeq
One can choose $\sig$ so that $\norm{\cH-H}_{\ka}<\al$. By construction, the cylinder
$\jC$ is now a good cylinder at energy $\e$ for $\cH$.
\end{proof}


\subsection{Cylinders with condition \pom}
We consider a proper Hamiltonian $H\in C^\ka(\A^3)$ and fix a regular value $\e$ of $H$.
We prove the following analog of Lemma~\ref{lem:goodcylfom} for condition \pomu.

\begin{lemma}\label{lem:goodcylpom}
Assume that $\jC\subset H\inv(\e)$ satisfies condition~\pomu.
Then for any $\al>0$, there exists a $C^\ka$ Hamiltonian  $\cH$
such that $\jC$ is a good cylinder for $\cH$, and which satisfies
\beq
\norm{\cH-H}_\ka<\al.
\eeq
\end{lemma}

The proof follows from a sequence of intermediate lemmas. First, by the same argument as in the
proof of the previous lemma, we can assume that
$\jC\subset H\inv(\e)$ admits a section $\Sig$ whose attached return map $\ph$ is a special
twist map. We set
\beq
\jC_*=\jA\cap H\inv(\e),\qquad \jC_*^\pm=j^\pm(\jC_*)
\eeq
so that $\jC\subset\jC_*$ and $\jC^\pm\subset\jC_*^\pm$.

\paraga The first lemma is a direct consequence of the last two conditions of \pomu, of which we keep
the notation.

\begin{lemma}\label{lem:existint}
For each $\jT\in\Tess(\jC)$, $\jT^+\cap\jT^-\neq\emptyset$.
\end{lemma} 

\begin{proof}
The Lipshitzian Lagrangian tori $\Psi^{sec}(\jT^+)$ and $\Psi^{sec}(\jT^-)$ are graphs over the null section
in $\jO^{sec}$. They moreover have the same
cohomology, since they are images of the same graph $\Psi^{ann}(\jT)$ by the exact-symplectic diffeomorphisms
$$
\Psi^{sec}\circ j^{\pm}\circ(\Psi^{ann})\inv.
$$
Hence their intersection in nonempty.
\end{proof}

\paraga Given a vector field $X$ on $\jC_*$ and a $2$-dimensional submanifold $S\subset \jC_*$,
we define the tangency set
$$
\Tan(X,S)=\big\{x\in S\mid X(x)\in T_xS\big\}.
$$
We say that a point $x\in S\setm \Tan(X,S)$ is regular.
We define the folds and cusps of $\jZ$ relatively to $X$ in the usual way (see \cite{GS}). 

\paraga We fix a compact subset $K\subset\De_\e$ which contains $\jC^+\cup\jC^-$ in its interior,
and we denote by $\H(K)$ the space of pairs of nonautonomous Hamiltonians in 
$C^\ka(\R\times \De_\e)$ with support in $[-2/3,1/3]\times K$. Let $\H_0$ be a ball centered at~$0$ 
in the space $\H(K)$, such that for each $(\ell_+,\ell_-)\in \H_0$, 
the conclusion of Lemma~\ref{lem:perturbpom} holds.
Given $(\ell_+,\ell_-)\in\H(K)$, we denote by $\cH_{(\ell_+,\ell_-)}$ the associated Hamiltonian, as defined
in Lemma~\ref{lem:perturbpom}.  Recall that $\cH_{(\ell_+,\ell_-)}$ coincides with the initial Hamiltonian
in the neighborhood of $\jA$. We can therefore assume that $\H_0$ is small enough so that $\jC$
still satisfies condition \pomu\ for $\cH_{(\ell_+,\ell_-)}$, relatively to the same section $\De$,
with characteristic maps
$$
j_{(\ell_+,\ell_-)}^\pm: \jA\to \jA^\pm(\ell_+,\ell_-)\subset\De.
$$
In order not to overload the notation and get rid of the problem of the boundaries and corners when considering intersections, 
we continue $\jC$ to a slightly larger $2$-dimensional manifold
{\em without boundary} contained in $\jC_*$ and with compact closure, {\em that we still denote by $\jC$}.
We set 
$$
\jC^\pm(\ell_+,\ell_-)=j_{(\ell_+,\ell_-)}^\pm(\jC),\qquad \jC_*^\pm(\ell_+,\ell_-)=j_{(\ell_+,\ell_-)}^\pm(\jC_*).
$$

\begin{lemma}\label{lem:genericprop}
The following properties are satisfied.
\begin{enumerate}
\item The set $\H_1$ of pairs of Hamiltonians  $(\ell_+,\ell_-)\in\H_0$ for which the intersection 
$$
\jC^+_*(\ell_+,\ell_-)\cap\jC^-_*(\ell_+,\ell_-)
$$ 
is transverse  in $\De_\e$ at each point of $\cl(\jC)$ is open and dense in $\H_0$. 
\item  For $(\ell_+,\ell_-)\in\H_1$,  the set
$$
\jI(\ell_+,\ell_-):=\jC^+(\ell_+,\ell_-)\cap\jC^-(\ell_+,\ell_-)
$$
is a two-dimensional submanifold  of $\De_\e$ which contains the set $\jX(\ell_+,\ell_-)$ of
all homoclinic intersections $j^+(\ell_+,\ell_-)(\jT)\cap j^+(\ell_+,\ell_-)(\jT)$ for $\jT\in\Tess(\jC)$.
\item We get rid of the indexation by $(\ell_+,\ell_-)$ when obvious. 
The set $\H_2$ of pairs of Hamiltonians  in $\H_1$ for which the subsets 
$$
\Big\{x\in \jI\mid X^-(x)\in T_xW^+(\jC)\Big\},\qquad
\Big\{x\in \jI\mid X^+(x)\in T_xW^-(\jC)\Big\},
$$
are 1-dimensional submanifolds of $\jI_*$, is open dense in $\H_0$. 
\item  The set $\H_3$ of pairs of Hamiltonians  in
$\H_2$ for which the tangency sets 
$$
\jZ^\pm=\Tan\big(X_\cH,\Pi^\pm(\jI_*)\big)
$$
are $1$-dimensional submanifolds of $\Pi^\pm(\jI_*)$, is open and dense in~$\H_0$.
\item The set $\H_4$ of pairs of Hamiltonians  in  $\H_3$ for which 
\beq
\jX \cap  (\Pi^+)\inv(\jZ^+)\cap (\Pi^-)\inv(\jZ^-)=\emptyset
\eeq
is open and dense in $\H_0$.
\item The set $\H_5$ of pairs of Hamiltonians  in  $\H_4$ for which each point 
of $\Pi^+\big(\jX\big)\cap\jZ^+$ and $\Pi^\pm\big(\jX\big)\cap\jZ^-$
is either regular or located on a fold, is open and dense set in $\H_0$.
\end{enumerate}
\end{lemma} 

\begin{proof} The proofs use only very classical methods of singularity theory, we will give the main ideas
and refer to \cite{LM} for more details.
We first recall the following classical genericity result by Abraham.
\vskip2mm\noindent
 {\bf Theorem (\cite{AR}).} {\em Fix $1\leq k<+\infty$. Let $\cL$ be a $C^k$ and second-countable 
 Banach manifold.
Let  $X$ and $Y$ be finite dimensional $C^k$ manifolds. 
Let $\chi : \cL\to C^k(X,Y)$ be a map such that the
associated evaluation 
$$
\ev_\chi: (\cL\times X)\to Y,\qquad \ev_\chi(\ell,x)=\big(\chi(\ell)\big)(x)
$$
is $C^k$ for the natural structures.
Fix a be a submanifold $D$ of $Y$ such that
$$
k>\dim X-\codim  D
$$
and assume that $\ev_\chi$ is transverse to $D$. 
Then the set of all $\ell\in\cL$ such that $\chi(\ell)$ is transverse to $D$
is residual in $\cL$, and open and dense when $D$ is compact.}

\vskip2mm $\bu$
To prove 1, we consider the following map
\beq
\left\vert
\begin{array}{rrl}
\chi:\H_0&\to &C^\ka(\jC_*\times \jC_*, \De_\e\times\De_\e)\\[4pt]
(\ell_+,\ell_-)\ \ &\mapsto& \Big[(x,y)\mapsto \big(j^+_{(\ell_+,\ell_-)}(x),j^-_{(\ell_+,\ell_-)}(y)\big)\Big]\\
\end{array}
\right.
\eeq
whose evaluation $\ev_\chi$ map is $C^{\ka-2}$. Introduce the diagonal 
\beq
D=\big\{(z,z)\mid z\in \De_\e\big\}\subset \De_\e\times\De_\e.
\eeq
Given $(\ell^0_+,\ell^0_-)\in\H_0$, a point $(x,y)\in \jC_*\times \jC_*$ such that 
$\ev_\chi\big((\ell^0_+,\ell^0_-),x,y\big)\in D$ and a vector $u\in T_y\De_\e$, 
one readily checks the existence of a path $s\mapsto (\ell_+,\ell_-)(s)$ such that 
\beq
(\ell_+,\ell_-)(s)=(\ell^0_+,\ell^0_-),\qquad \frac{d}{ds}\ev_\chi(\ell(s),x,y)_{\vert s=0}=(0,u),
\eeq
which proves that $\ev_\chi$ is transverse to $D$ at $(\ell,x,y)$. For $\ka$ large enough, 
the Abraham genericity theorem 
proves our first claim, while 
the second one is an immediate consequence of Lemma~\ref{lem:existint}.

\vskip2mm $\bu$
To prove 3, we endow $\A^3$ with a trivial Riemannian metric, and we denote by $\cdot$
the associated scalar product, which is therefore defined for pairs of vectors possibly
tangent at different points.
Given $x\in \jA^\pm$, we denote by $N^\pm(x)$ the (suitably oriented) unit normal vectors at $x$  to
the invariant manifolds $W^\pm(\jA)$. We then introduce the map (where we get rid of the obvious
 indexation of $j^\pm$ by $(\ell_+,\ell_-)$):
\beq
\left\vert
\begin{array}{rrl}
\chi:\H_0&\to &C^\ka(\jC_*\times \jC_*, \De_\e\times\De_\e\times\R)\\[4pt]
(\ell_+,\ell_-)\ \ &\mapsto& 
\Big[(x,y)\mapsto \Big(j^+(x),j^-(y),X^-\big(j^-(x)\big)\cdot N^+\big(j^+(y)\big)\Big)\Big]\\
\end{array}
\right.
\eeq
whose evaluation $\ev_\chi$ map is $C^{\ka-2}$. One then proves by straightforward constructions
that $\ev_\chi$ is tranvsverse to the diagonal 
$$
D=\{(\xi,\xi,0)\mid \xi\in\De_\e\}.
$$
As a consequence of the Abraham theorem, for an open dense subset of pairs $(\ell_+,\ell_-)$,
$\ha\chi:=\chi(\ell_+,\ell_-)$ is transverse to $D$. For such pairs,
$
\ha\chi\inv(D)
$
is a codimension $5$ submanifold of $\jC_*\times\jC_*$, and so is $1$-dimensional. The
projection of this manifold on the first factor is the set of points.

\vskip2mm $\bu$ The proof of 4 is analogous. We denote by $Y^\pm$ the direct image of
the vector field $X_H$ restricted to $\jA$ by the transition maps $j^\pm$, which is therefore
a vector field on $\jA^\pm$. We introduce the projection operator $P$ on the normal bundle
of the intersection $\jI$, that we see as a submanifold of the trivial bundle $T\A^3$. 
To deal with $\jZ^-$, we introduce the following map
\beq
\left\vert
\begin{array}{rrl}
\chi:\H_0&\to &C^\ka(\jC_*\times \jC_*, \De_\e\times\De_\e\times\R)\\[4pt]
(\ell_+,\ell_-)\ \ &\mapsto& 
\Big[(x,y)\mapsto \Big(j^+(x),j^-(y),\norm{P_{j^+(x)}(Y^-(y))}^2\Big)\Big]\\
\end{array}
\right.
\eeq
which yields the same result as above.

\vskip2mm $\bu$ To prove 5 we begin by proving that the set of pairs $(\ell_+,\ell_-)\in\H_4$ 
for which 
$$
\jX \cap  (\Pi^+)\inv(\jZ^+)\cap (\Pi^-)\inv(\jZ^-)
$$
is finite is open and dense in $\H_0$. For this we introduce the map
$$
\chi:\H_0\to C^\ka(\jC_*\times \jC_*, \De_\e\times\De_\e\times\R\times\R)
$$
such that
$$
\big(\chi(\ell_+,\ell_-)\big)(x,y)=
\Big(
j^+(x),
j^-(y),
Y^-\big(j^-(x)\big)\cdot N^+\big(j^+(y)\big),
Y^+\big(j^+(x)\big)\cdot N^-\big(j^-(y)\big)
\Big)
$$
which is also easily proved to be transverse to the diagonal
$$
D=\big\{(\xi,\xi,0,0)\mid \xi\in\De_\e\big\}.
$$
So the set of pairs for which $\ha\chi\inv(D)$ is a dimension 0 manifold is open and dense,
which proves our claim. It only remains to check that one can make an additional perturbation
which disconnects $\jX$ from this last set, which is easy.

\vskip2mm $\bu$ The proof of 6 is analogous to the previous one, since the set of cusp
points is finite.
\end{proof}

\paraga The following lemma immediately yields Lemma~\ref{lem:goodcylpom}. Recall that
$\jC$ admits a section $\Sig$ for which the return map $\ph$ is a special twist map.

\begin{lemma}
Given $(\ell^0_+,\ell^0_-)\in\H_5$, with associated Hamiltonian $\cH^0$, 
there exists $\cH$ in $C^\ka(\A^3)$, arbitrarily close to $\cH^0$ in the $C^\ka$ topology and
which coincides with $\cH^0$ in the neighborhood of $\jA$, for which:
\begin{itemize}
\item $\jC$ satisfies condition \pomu\ and admits a homoclinic correspondence $\psi$;
\item for any essential circle $\Ga\in\Ess(\ph)$ there exists a splitting arc based on $\Ga$
for $(\ph,\psi)$; 
\item if moreover $\Ga$ is the lower boundary of a Birkhoff zone, there exists
a simple arc based of $\Ga$ for $(\ph,\psi)$.
\end{itemize}
\end{lemma}

\begin{proof} 
We will first prove that for any $\Ga\in\Ess(\ph^0)$ there exists a splitting
arc based on $\Ga$ for $(\ph^0,\psi^0)$. Let $\jT\subset\jC$ be the torus generated by 
$\Ga$ under the action of the Hamiltonian flow. The complement $\jC\setm\jT$ admits
two connected components, which we denote by $\jC_\bu$ and $\jC^\bu$ according
to the orientation of $\Sig$ induced by its parametrization.

\vskip2mm $\bu$ Set $I=\jT^+\cap \jT^-$, so that $I$ is compact and contained in $\jX$.
By property 5) of Lemma~\ref{lem:genericprop}, for each point $\xi\in I$, one of the points
$\xi^\pm=\Pi^\pm(\xi)$ is regular with respect to $X_H$, in the sense that $X_H$ is not
tangent to $\Pi^\pm(\jI)$ at $\xi^\pm$. We say that $\xi$ is positively (resp. negatively) regular
when $\xi^+$ (resp. $\xi^-$) is regular.

We define the point $x^+$ as the first intersection
of the positive orbit of $\xi^+$ with $\Sig$, and the point $x^-$ as the first intersection of
the negative orbit of $\xi^-$ with $\Sig$. We adopt the same convention for the transport
to $\Sig$ for all points in small enough neighborhoods of the previous two ones.

\vskip2mm $\bu$   Assume for instance that $\xi$ is positively regular. 
Then there exists a neighborhood $V(\xi)$ of 
$\xi$ in $\jI$ such that the previous transport by the Hamiltonian flow in $\jC$ induces an embedding of 
the image $\Pi^\pm(V(\xi))$ in the section $\Sig$, whose image $V^+$ is a neighborhood of  
the point $x^+$.  Hence one can moreover assume $V(\xi)$ small enough so that the complement 
$\ha\Ga=V^+\setm\Ga$ admits exactly two connected components, contained in $\jC^\bu$
and $\jC_\bu$. Inverse transport of $\ha\Ga$ by the Hamiltonian flow and then by $j^+$ yields 
a Lipshitzian curve in $V(\xi)$, which disconnects $V(\xi)$. We denote by $V^\bu(\xi)$ and 
$V_\bu(\xi)$ the two components of its complement, according to the intersections
of their images with $\jC_\bu$ and $\jC^\bu$. We say that $\jC^\bu$ is the positive component
and that $\jC_\bu$ is the negative component.

\vskip2mm $\bu$  Then, the point $\xi^-$ is either regular or located on a fold of $\Pi^-(\jI)$ 
relatively to $X_H$. One can reduce $V(\xi)$ in order that the intersection $\Pi^-(V(\xi))\cap \jT$
is connected, and therefore arc-connected. We will assume this condition satsified in the following.
This yields by inverse transport an arc-connected subset $\sig=j^-(\sig^-)\subset V(\xi)$. 
We say that the intersection $I$ is {\em one-sided in $V(\xi)$} if either
$
\sig\cap V_\bu(\xi)=\emptyset
$
or 
$
\sig\cap V^\bu(\xi)=\emptyset
$
is empty, we say that is is positive in the latter case and negative in the former.

\vskip2mm $\bu$ In the case where $\xi$ is negatively regular, define the neighborhood $V(\xi)$ in a symmetric way,
with the analogues for the notion of one-sided, positive and negative intersections in $V(\xi)$. When both $\xi^+$
and $\xi^-$ are regular, we arbitrarily choose one of them to perform the previous construction
and define the notion of one-sided intersection.

\vskip2mm $\bu$ One then gets a covering of $I$ by a finite number of neighborhoods $V(\xi_1),\ldots,V(\xi_\ell)$
with the previous properties. Arguing by contraction and assuming that each positively regular point $\xi_i$
yields a positive intersection and each negatively regular point yields a negative intersection, one proves that
it would be possible to construct a pair of arbitrarily close to the identity Hamiltonian diffeomorphisms $\phi^\pm$ of $\De_\e$
to which Lemma~\ref{lem:perturbpom} applies, and which yield a $C^\ka$ perturbation $\cH$ of $H$ which 
still admits $\jC$ as a cylinder satisfying \pomu, and $\jT$ as an invariant torus, but for which $\jT^+\cap\jT^-$ is
empty, which is impossible.

\vskip2mm $\bu$ Therefore there exists in $I$ either a positively regular point with two-sided or negative
intersection, or a negatively regular point with two-sided or positive intersection. Both cases yield a splitting
arc based on $\Ga$ for the (local) composition $P^+\circ\Pi^+\circ j-\circ(P^-)\inv$. Now we know that 
if for instance $\xi$ is positively regular, then $\xi^-$ is either regular or located on a fold. This immediately
yields a splitting domain for the homoclinic correpondence and concludes the proof.
\end{proof}

\paraga It only remains now to consider the case of lower boundaries of Birkhoff zones.

\begin{lemma}
If $\Ga$ is the lower boundary of a Birkhoff zone, there exists a splitting arc based on $\Ga$.
\end{lemma}

\begin{proof}
The main observation is that the set of essential circles which are lower boundaries of Birkhoff zones
in countable. Therefor it is enough to prove the property for one such circle. But this is an immediate
perturbation of the previous lemma, by construction a pair of Hamiltonian diffeomorphisms which
makes the intersection point of the previous splitting arc be a derivability point for the arc, with a 
non-vertical tangent, and a derivability point for the circle, with a tangent distinct from the tangent
to the arc.
\end{proof}


\subsection{From admissible chains to good chains: proof of Theorem~\ref{thm:main2}}\label{ssec:proofmain2}
Theorem~\ref{thm:main2} is now a direct consequence of the previous two sections.
Consider an admissible chain $(\jC_k)_{1\leq k\leq k_*}$ with $\de$-bounded homoclinic correspondence
at energy $\e$ for $H$.  Then by Lemma~\ref{lem:goodcylfom} used
recursively there exists a Hamiltonian $\cH^0$ arbitrarily $C^\ka$-close to $H$ such that each 
cylinder satisfying \fomu\ is a good cylinder, and such that each gluing boundary still admits transverse
homoclinic intersections. By Lemma~\ref{lem:goodcylpom} used recursively, one can then find $\cH^1$
arbitrarily $C^\ka$-close to $\cH^0$ such that each cylinder \pomu\  is a good cylinder, without altering
the heteroclinic condition \pomd\ or the gluing condition \glu. 

The only remaining point to prove is that one perturb the system in such a way that each bifurcation pair 
$(\jC_k,\jC_{k+1})$ admits a bifurcation map in the sense of Section~\ref{sec:goodcylchains}, paragraph 6.
We keep the assumptions and notation of \cite{Mar}, Part I, Section 2.
We can first make family of small $C^\ka$ perturbations to the Hamiltonian $H_\eps=h+\eps f$, parametrized by $\eps$, 
to make it $C^\infty$ in the neighborhood of $\T^3\times\{b\}$. 
We so construct a family of $C^\infty$ Hamiltonians $\cH_\eps$ such that $\norm{\til H_\eps-H_\eps}_{C^p}\leq \al\eps$, 
which we still write
$$
\til H_\eps=\til h_\eps+\eps\til f_\eps=h+\eps \ha f_\eps,
$$
so that $\ha f_\eps$ is $C^p$ $\al$-close to $f$. We then perturb $\til f_\eps$ in order that its bifurcation point admits a 
frequency vector $\om$ which is $2$-Diophantine, that is, in adapted coordinates: 
$\om=(\ha\om,0)$ with $\ha\om$ Diophantine. One then use the $\eps$-dependent normal form of $\H_\eps$
introduced in \cite{Mar}, Appendix 3. Given two constants $d>0$ and $\de<1$ with $1-\de>d$,
then,  there is an $\eps_0>0$ such that for $0<\eps<\eps_0$, there exists an analytic symplectic embedding 
$$
\Phi_\eps: \T^3\times B(b_1,\eps^d)\to \T^3\times B(b_1,2\eps^d)
$$
such that 
$$
\til N_\eps(\th,r)=\til H_\eps\circ\Phi_\eps(\th,r)=\til h_\eps(r)+ g_\eps(\th_3,r)+R_\eps(\th,r),
$$
where $g_\eps$ and $R_\eps$ are $C^p$ functions and
\beq
\norm{R_\eps}_{C^p\big( \T^n\times B(b_1\eps^d)\big)}\leq \eps^3.
\eeq
Moreover, $\Phi_\eps$ is close to the identity, in the sense that
\beq
\norm{\Phi_\eps-\Id}_{C^p\big( \T^3\times B(b_1,\eps^d)\big)}\leq \eps^{1-\de}.
\eeq
The final perturbation of the initial Hamiltonian (localized in the neighborhood of the bifurcation locus)
 will be the truncation
$$
H_\eps^0=(\til h_\eps+g_\eps)\circ\Phi_\eps\inv=h+\eps f^0_\eps
$$
which is $\al$ close to $H_\eps$ when $\eps$ is small enough. One  checks that $H_\eps^0$
satisfies the condition of a good chain for the cylinders at their bifurcation point and that 
$f^0_\eps$ is $\al$ close to $f$ is $\eps$ is small enough. This concludes our proof.


\subsection{Proof of Theorem~\ref{thm:main1}}\label{ssec:proofmain1}
Fix $\e>\Min h$ together with a finite family of arbitrary open sets $O_1,\ldots,O_m$
which intersect $h\inv(\e)$. By \cite{Mar}, for $\ka\geq \ka_0$ large enough, there exists a 
lower-semicontinuous function 
$$
\beps_0:\jS^\ka\to\R^+
$$ 
with positive values on an open dense subset of $\jS^\ka$ such that for $f\in\jB^\ka(\beps_0)$
the system
\begin{equation}
H(\th,r)=h(r)+ f(\th,r)
\end{equation}
admits an admissible chain $(\jC_k)_{1\leq k\leq k_*}$, with $(\de/2)$-bounded homoclinic
correspondence, such that each open set $\T^3\times O_k$ contains the $\de$-neighborhood
in $\A^3$ of some essential torus of the chain.  By Theorem~\ref{thm:main2} there exists $\cH$,
arbitrarily $C^\ka$-close to $H$, such that $(\jC_k)_{1\leq k\leq k_*}$ is a good chain with
$\de$-bounded homoclinic correspondence. By \cite{GM} there exists an orbit for $\cH$
which is $\de$-admissible, and therefore which intersects each $\T^3\times O_i$. This
last property being open, Theorem~\ref{thm:main1} is proved.

\newpage

\appendix

\section{Normal hyperbolicity and symplectic geometry}\label{app:normhyp}
\setcounter{paraga}{0}

We refer to \cite{Berg10,BB13,C04,C08,HPS} for the references on normal hyperbolicity.
 Here we limit ourselves to a very simple class of systems which
admit a normally hyperbolic invariant (non compact) submanifold, which serves us as a model from which
all other definitions and properties will be deduced.

\paraga  The following statement is a simple version of the persistence theorem for normally 
hyperbolic manifolds well-adapted to our setting, whose germ can be found in \cite{B10} and whose proof
can be deduced from the previous references.

\vskip3mm

\noindent {\bf The normally hyperbolic persistence theorem.} 
{\em Fix $m\geq 1$ and consider a vector field on $\R^{m+2}$ of the form 
$\jV=\jV_0+\jF$,
with $\jV_0$ and $\jF$ of class $C^1$ and reads
\begin{equation}\label{eq:formV0}
\dot x=X(x,u,s),\qquad \dot u=\la_u(x)\, u,\qquad \dot s =-\la_s(x)\,s,
\end{equation}
for $(x,u,s)\in \R^{m+2}$. Assume moreover that there exists $\la>0$ such that
the inequalities
\beq\label{eq:ineg}
\la_u(x)\geq \la,\quad{\it and}\quad \la_s(x)\geq \la,\qquad  x\in\R^m.
\eeq
hold. Fix a constant $\mu>0$.
Then there exists a constant $\de_*>0$ such that if 
\begin{equation}\label{eq:condder}
\norm{\partial _xX}_{C^0(\R^{m+2})}\leq \de_*,\qquad 
\norm{\jF}_{C^1(\R^{m+2})}\leq \de_*,
\end{equation}
the following assertions hold.
\begin{itemize}
\item The maximal invariant set for $\jV$  contained in $O=\big\{(x,u,s)\in\R^{m+2}\mid \norm{(u,s)}\leq \mu\big\}$
is an $m$-dimensional manifold $\Ann(\jV)$ which admits the graph representation:
$$
\Ann(\jV)=\big\{\big(x,u=U(x), s=S(x)\big)\mid x\in\R^m\big\},
$$ 
where $U$ and $S$ are $C^1$ maps $\R^m\to\R$  
such that \begin{equation}\label{eq:loc}
\norm{(U,S)}_{C^0(\R^m)}\leq \frac{2}{\la}\,\norm{\jF}_{C^0}.
\end{equation}
\item The maximal positively invariant set for $\jV$  contained in  $O$ 
is an $(m+1)$-dimensional manifold $W^+\big(\Ann(\jV)\big)$ which admits the graph representation:
$$
W^+\big(\Ann(\jV)\big)=\big\{\big(x,u=U^+(x,s), s\big)\mid x\in\R^m\ s\in[-\mu,\mu]\big\},
$$ 
where $U^+$ is a $C^1$ map $\R^m\times[-1,1]\to\R$  
such that \begin{equation}\label{eq:loc2}
\norm{U^+}_{C^0(\R^m)}\leq c_+\,\norm{\jF}_{C^0}.
\end{equation}
for a suitable $c_+>0$.
Moreover, there exists $C>0$ such that for $w\in W^+\big(\Ann(\jV)\big)$, 
\beq
\dist\big(\Phi^t(w),\Ann(\jV)\big)\leq C\exp(-\la t),\qquad t\geq0.
\eeq
\item The maximal negatively invariant set for $\jV$  contained in $O$
is an $(m+1)$-dimensional manifold $W^-\big(\Ann(\jV)\big)$ which admits the graph representation:
$$
W^-\big(\Ann(\jV)\big)=\big\{\big(x,u, s=S^-(x,u)\big)\mid x\in\R^m,\ u\in[-\mu,\mu]\big\},
$$ 
where $S^-$ is a $C^1$ map $\R^m\times[-1,1]\to\R$  
such that \begin{equation}\label{eq:loc3}
\norm{S^-}_{C^0(\R^m)}\leq c_-\,\norm{\jF}_{C^0}.
\end{equation}
for a suitable $c_->0$.
Moreover, there exists $C>0$ such that for $w\in W^-\big(\Ann(\jV)\big)$, 
\beq
\dist\big(\Phi^t(w),\Ann(\jV)\big)\leq C\exp(\la t),\qquad t\leq0.
\eeq
\item
The manifolds $W^\pm\big(\Ann(\jV)\big)$ admit $C^0$ foliations $\big(W^\pm(x)\big)_{x\in \Ann(\jV)}$ such that 
for $w\in W^\pm(x)$
\beq
\dist\big(\Phi^t(w),\Phi^t(x)\big)\leq C\exp(\pm\la t),\qquad t\geq0.
\eeq
\item If  moreover $\jV_0$ and $\jF$ are of class $C^p$, $p\geq1$, and if in addition of the 
previous conditions the domination inequality, the condition
\beq\label{eq:addcond}
p\,\norm{\partial_x X}_{C^0(\R^m)}\leq \la
\eeq
holds,  then the functions $U$, $S$, $U^+$, $S^-$ are of class $C^p$ and 
\beq
\norm{(U,S)}_{C^p(\R^m)}\leq C_p \norm{\jF}_{C^p(\R^{m+2})}.
\eeq
for a suitable constant $C_p>0$.
\item Assume moreover that the vector fields $\jV_0,\jV$ are $R$-periodic in $x$, where $R$ is a lattice in $\R^m$.
Then their flows and the manifolds $\Ann(\jV)$ and $W^\pm\big(\Ann(\jV)\big)$ pass to the quotient $(\R^m/R)\times \R^2$
Assume that the time-one map of $\jV_0$ on $\R^m/R\times\{0\}$ is $C^0$ bounded by a constant $M$.
Then, with the previous assumptions, the constant $C_p$ depends only on $p$, $\la$ and $M$.
\end{itemize}
}

The last statement will be applied in the case where $m=2\ell$ and $R=c\Z^\ell\times\{0\}$, where $c$ is a positive constant,
so that the quotient $\R^{2\ell}/R$ is diffeomorphic to the annulus $\A^\ell$.

\paraga The following result describes the symplectic geometry of our system in the case where $\jV$
is a Hamiltonian vector field. We keep the notation of the previous theorem.

\vskip3mm

\noindent {\bf The symplectic normally hyperbolic persistence theorem.}  {\it Endow $\R^{2m+2}$ with a symplectic form
$\Om$ such that there exists a constant $C>0$ such that  for all $z\in O$
\beq\label{eq:assumpsymp}
\abs{\Om(v,w)}\leq C\norm{v}\norm{w},\qquad \forall v,w \in T_z M.
\eeq
Let $\jH_0$ be a $C^2$ Hamiltonian on $\R^{2m+2}$ whose Hamiltonian vector field $\jV_0$ satisfies (\ref{eq:formV0})
with conditions (\ref{eq:ineg}), and consider a Hamiltonian $\jH=\jH_0+\jP$.
Then there exists a constant $\de_*>0$ such that if 
\begin{equation}\label{eq:condder2}
\norm{\partial _xX}_{C^0(\R^{m+2})}\leq \de_*,\qquad 
\norm{\jP}_{C^2(\R^{m+2})}\leq \de_*,
\end{equation}
the following properties hold.
\begin{itemize}
\item The manifold $\Ann(\jV)$ is $\Om$-symplectic.
\item The manifolds $W^\pm\big(\Ann(\jV)\big)$ are coisotropic and the center-stable and center-unstab\-le
foliations $\big(W^\pm(x)\big)_{x\in \Ann(\jV)}$ coincide with the characteristic foliations 
of the manifolds $W^\pm\big(\Ann(\jV)\big)$.
\item If $\jH$ is $C^{p+1}$ and condition (\ref{eq:addcond}) is satisfied, then
$W^\pm\big(\Ann(\jV)\big)$ are of class $C^p$ and the foliations $\big(W^\pm(x)\big)_{x\in \Ann(\jV)}$
are of class $C^{p-1}$.
\item There exists a neighborhood $\jO$ of $\Ann(\jV)$ and a symplectic straightening symplectic diffeomorphism 
$\Psi:\jO\to O$ such that
\beq
\begin{array}{lll}
\Psi\big(\Ann(\jV)\big)=\A^{\ell}\times\{(0,0)\};\\[4pt]
\Psi\big(W^-\big(\Ann(\jV)\big)\big)\subset\A^{\ell}\times\big(\R\times\{0\}\big),\qquad
\Psi\big(W^-\big(\Ann(\jV)\big)\big)\subset\A^{\ell}\times\big(\{0\}\times\R\big);\\[4pt]
\Psi\big(W^-(x)\big)\subset\{\Psi(x)\}\times\big(\R\times\{0\}\big),\qquad
\Psi\big(W^+(x)\big)\subset\{\Psi(x)\}\times\big(\{0\}\times\R\big).\\
\end{array}
\eeq
\end{itemize}
}

See \cite{Mar} for a proof.

\newpage

%
%
%
%

%


\end{document}